\documentclass[10pt]{amsart}

\usepackage{amsfonts,amsmath,latexsym,amssymb,verbatim,amsbsy,amsthm}
\usepackage{dsfont,bm}

\usepackage[top=1in, bottom=1in, left=1in, right=1in]{geometry}

\usepackage[dvipsnames]{xcolor}
\usepackage[colorlinks=true, pdfstartview=FitV, linkcolor=RoyalBlue,citecolor=ForestGreen, urlcolor=blue]{hyperref}

\theoremstyle{plain}
\newtheorem{THEOREM}{Theorem}[section]

\newtheorem{theorem}[THEOREM]{Theorem}

\newtheorem{lemma}[THEOREM]{Lemma}

\theoremstyle{definition}

\theoremstyle{remark}

\newtheorem{remark}[THEOREM]{Remark}

\def \a {\alpha}

\def \g {\gamma}
\def \d {\delta}
\def \e {\varepsilon}
\def \f {\varphi}
\def \k {\kappa}
\def \l {\lambda}
\def \n {\nabla}
\def \s {\sigma}

\def \D {\Delta}

\def \L {\Lambda}

\def \O {\Omega}

\def \cD {\mathcal{D}}

\def \cH {\mathcal{H}}
\def \cI {\mathcal{I}}

\newcommand{\R}{\ensuremath{\mathbb{R}}}   
\newcommand{\I}{\ensuremath{\mathbb{I}}}   

\def \lan {\langle}
\def \ran {\rangle}
\def \llan {\left \langle}
\def \rran {\right \rangle}
\def \p {\partial}
\def \ra {\rightarrow}
\def \ss {\subset}

\def \HI {H\"older inequality}

\renewcommand{\ge}{\geqslant}
\renewcommand{\leq}{\leqslant}
\renewcommand{\le}{\leqslant}

\def \dk  {\, \mbox{d}\kappa}
\def \drho  {\, \mbox{d}\rho}
\def \dx  {\, \mbox{d}x}

\def \dmu  {\, \mbox{d}\mu}

\def \dv  {\, \mbox{d} v}

\def \ddt  {\frac{\mbox{d\,\,}}{\mbox{d}t}}

\def \dd  {\mbox{d}}

\def \dB {\dot{B}}

\def \dom {{\O \times \R^n}}
\def \uave {[u]_{\rho}}

\def \bh {\bar{h}}

\def \one {{\mathds{1}}}
\newcommand{\aro}[1]{[#1]_{\rho}}
\begin{document}

\title[Exponential relaxation to equilibrium for a kinetic Fokker-Planck-Alignment equation]{Exponential relaxation to equilibrium for a kinetic Fokker-Planck-Alignment equation with force}

\author{Vinh Nguyen}

\address{ Department of Mathematics, Statistics and Computer Science, University of Illinois at Chicago, 851 S Morgan St, Chicago, IL 60607}

\email{vnguye66@uic.edu, vinhtnguyen.math@gmail.com}




\thanks{\textbf{Acknowledgment.} The author would like to thank Professor Roman Shvydkoy for useful discussions and acknowledge partial support from NSF grant DMS-2107956 (PI: Roman Shvydkoy).  
	}

\begin{abstract}
In this note, we consider a kinetic Fokker-Planck-Alignment equation with Rayleigh-type friction and self-propulsion force which is derived from general environmental averaging models. We show the exponential relaxation in time toward equilibrium of the solutions provided certain spectral gap conditions are satisfied. The result is proved by using Desvillettes-Villani's method for collisional models to establish the global hypocoercivity.
\end{abstract}

\maketitle

\section{Introduction}
In this note, we are interested in a kinetic Fokker-Plank-Alignment equation which is derived from general environmental averaging models. More specifically, let $\O\ss\R^n$ be a periodic domain. An agent is  featured by its position $x\in \O$ and its velocity $v\in \R^n$. The density of agents who has position $x$ and velocity $v$ at time $t \ge 0$, denoted by $f = f(x,v,t)$, is governed by the following equation:
\begin{equation}\label{e:f}
\p_t f + v\cdot \n_x f = s_\rho\big[ \D_v f + \n_v\cdot \big(\big[(v - [u]_{\rho}) +  F(v)\big] f\big)\big],  
\end{equation}
subject to the initial condition
\begin{equation*}
    f(x,v,0) = f_0(x,v).
\end{equation*}
Here $\rho$ and $u$ are macroscopic density and macroscopic velocity defined by
\begin{equation}
    \rho (x) =\int_{\R^n} f(x,v) \dv, \quad u\rho(x) = \int_{\R^n} v f(x,v) \dv.
\end{equation}
The family of pairs $(\k_\rho, \aro{\cdot})$ with $ \dk_\rho := s_\rho \drho$ satisfies the conditions for a material environmental averaging model introduced in \cite{RS2023env}.
The Rayleigh-type friction and self-propulsion force $F$ is given by
\begin{equation}\label{force}
F(v) = \dfrac{\s(|v|^p - 1)v}{\eta(|v|)},
\end{equation}
where $\eta: \R_+ \ra \R_+$ is a smooth, positive and increasing function satisfying
\begin{equation}\label{eta}
   \eta(z) = 1 \text{ if } z \le R \text{ for some } R > 0;\text{ and }  \eta(z) \sim z^q \text{ for some } q > p \text{ as } z \ra \infty.
\end{equation}
Our goal is to show that the solution of \eqref{e:f} relaxes exponentially fast toward its equilibrium. We utilize the Desvillettes-Villani's method (see \cite{DesVil05, villani2009hypo}) for collisional models to modify the entropy and establish a global hypocoercivity. Without additional force, Shvydkoy gave the first result on global hypocoercivity for this type of model in \cite{RS2022hypo}. In that paper, the averaging operator is given by
\[
\aro{u} := \phi *\big(\frac{\phi*(u\rho)}{\phi *\rho}\big),
\]
where $\phi$ is a radial non-negative non-increasing function satisfying 
 \begin{equation*}
     \int_{\O} \phi(x)\dx = 1, \quad \phi (x) \ge c_0 \one _{\{|x| < r_0\}}.
 \end{equation*}
Then the result was extended to a class of kinetic equations in \cite{RS2023env}. In this work, we show that if an extra force is added then we still have a global hypocoercivity and hence, an exponential relaxation to equilibrium provided that the force is small in the sense of assumption (iv) below.

\medskip
Before stating our result, let us give some motivation for studying the equation \eqref{e:f}. The study of collective behavior has attracted a lot of attention from the scientific community because it has diverse applications ranging from biology, physics, computer science, social science etc., see e.g. \cite{CamEtAl2003,  Sbook, VZ2012, yates2009a} and the references therein.

For microscopic descriptions, many models of collective behavior can be described as follows:
\begin{equation}\label{abs-form}
    \begin{cases}
        \dot{x}_i = v_i, \quad\quad (x_i, v_i) \in \O\times \R^n,\\
        \dot{v}_i = s_i ([v]_i - v_i)+ F_i, \quad i = 1, \ldots, N,
    \end{cases}
\end{equation}
where  $s_i, F_i$ are respectively the communication strength and the force corresponding to the $i$-th agent; $v = (v_1, \ldots, v_N)\R^{nN}$ and $[v]_i$ denotes the averaging operator acts on the $i$-th agent. The celebrated Cucker-Smale system \cite{CS2007a, CS2007b} can be written in form \eqref{abs-form} with
\begin{equation}\label{CS}
s_i = \sum_{j=1}^N m_j \phi(|x_i-x_j|), \qquad  [v]_i = \dfrac{\sum_{j=1}^N m_j \phi(|x_i-x_j|) v_j}{\sum_{j=1}^N m_j \phi(|x_i-x_j|)},
\end{equation}
where $\phi$ is a smooth radial non-increasing function, $m_i$ is the communication weight of the $i-$th agent. In this model $F_i = 0$. For examples with nontrivial force $F_i$, the readers can see \cite{LRS-friction, ShuTa2019, Sbook}. If we take $F_i$ in \eqref{abs-form} to be the combination of a deterministic force and a noise of the form
\begin{equation}\label{force-agent}
F_i = \dfrac{\s(1-|v_i|^p )v_i}{\eta(|v_i|)} + \sqrt{2 s_i(x)}\dB_i, \quad 0 < \s < 1 \text{ and }  p>0,
\end{equation}
here $\eta$ is given by \eqref{eta} and $B_i's$ are independent Brownian motions in $\R^n$, then the stochastic mean-field limit of \eqref{abs-form} formally leads to the kinetic equation \eqref{e:f}.

\medskip
In this short note, we will merely focus on the long-time behavior of the solution of \eqref{e:f} provided it exists. For a rigorous derivation of \eqref{e:f} via stochastic mean-field limit one can consult the scheme from \cite{BCC2011,  RS2023env}. For the existence of solution, we refer to \cite{BCC2011, KarperEtAl2013, RS2023env}. We assume the solution $f$ to \eqref{e:f} belongs to some weighted Sobolev space
\begin{equation*}
    H^k_l(\O\times \R^n):= \left\{f : \sum_{k'\le k}\sum_{|\a| = k'} \int_{\O\times\R^n} \llan v\rran^{l+ 2(k-k')}|\p^\a_{x,v} f|^2 \dx\dv < \infty\right\},
\end{equation*}
where $\llan v\rran = \sqrt{1+ |v|^2}$ and $\a$ denotes a multiindex.

Next let us introduce some more notations. Letting $G: \R_+ \ra \R$ be the function defined by
\[
G(z) := \int_0^z \dfrac{\s (y^{p+1} - y)}{\eta(y)} \mbox{ d}y,
\]
and letting
\begin{equation}\label{V}
    V(v) = \dfrac{|v|^2}{2} + G(|v|).
\end{equation}
Then the gradient and Hessian matrix of $V$ can be computed explicitly,
\begin{align}
    \n V &= v +  F(v),\label{gradV}\\
    \n^2 V & =\left(1+ \frac{\s(|v|^p - 1)}{\eta(|v|)} \right)\I + \frac{\s |v|^{p}}{\eta(|v|)} \frac{v}{|v|}\otimes \frac{v}{|v|} -\frac{\s (|v|^p -1) |v|\eta'(|v|)}{\eta^2(|v|)} \frac{v}{|v|}\otimes \frac{v}{|v|}, \label{HessV}
\end{align}
where $\I$ is the identity matrix.
\begin{remark}
By the assumption \eqref{eta} and the identity \eqref{HessV} we see that the Hessian matrix of $V$ is bounded. Thus, there exists a positive constant $\L$ such that 
\begin{equation}\label{HessV-bdd}
    |(\n^2 V) (y)| \le \L |y|, \quad \forall y \in \R^n.
\end{equation}
We also note that for $y \in \R^n$,
\begin{equation}\label{HessV-coer}
  y^T (\n^2 V) y \ge \left(1 - \frac{\s}{\eta(|v|)} -\frac{\s |v|^{p+1}\eta'(|v|)}{\eta^2(|v|)} \right)|y|^2 \ge \l |y|^2,
\end{equation}
where $\l >0$ is a constant depending on $\s$.
\end{remark}  
We expect that the solution to \eqref{e:f} converges to 
\begin{equation}\label{limit}
    f_{\infty} :=  \dfrac{1}{Z} e^{-V(v)}\quad \text{ with } Z = \int_{\dom} e^{-V(v)} \dv\dx.
\end{equation}
The macroscopic field $ u_F$ is  defined by
\begin{equation*}
    \rho u_F(x) = \int_{\R^n} F(v) f(x,v) \dv.
\end{equation*}
Denote $L^2 (\k_\rho) := L^2(\dd \k_\rho)$. The inner product in $L^2(\k_\rho)$ is denoted by $\llan \cdot, \cdot\rran_{\k_{\rho}}$. Our main result is the following:
\begin{theorem}\label{thm-main}
Suppose that $f \in H^k_l(\dom)$ is a solution to \eqref{e:f} such that $\rho (t)$ satisfies the following assumptions for all $t\ge 0$:
\begin{enumerate} 
    \item[(i)]  $c_0\le s_\rho\leq c_1$ and $\|\n s_\rho\|_{\infty} \leq c_2$, where $ c_0, c_1, c_2$ are positive constants, 
    \item[(ii)] $\n_x (s_\rho [\cdot]_\rho): L^2(\rho)\ra L^2(\rho)$ is uniformly bounded,
    \item[(iii)] 
    there exists a constant $0<\e_0<1$ such that 
    \[
    \sup \big\{\llan w,[w]_{\rho}\rran_{\k_\rho} | \; w \in L^2(\k_\rho), \|w\|_{L^2(\k_\rho)} = 1\big\}    \le 1-\e_0,
    \]
    \item[(iv)] there exists a constant $0< \e_1 < 1$ such that
    \[
    \|u_F\|_{L^2(\k_\rho)} \le\e_1\|u\|_{L^2(\k_\rho)}.
    \]    
\end{enumerate}
Then $f$ converges to $f_{\infty}$ exponentially fast:
\begin{equation*}
    \|f(t) - f_{\infty}\|_{L^1(\dom)} \le C e^{-\d t},
\end{equation*}
where $C>0$ is a constant depending on initial data $f_0$ and given parameters; $\d>0$ is a constant depending only on given parameters. 
\end{theorem}
\begin{remark}
    Observe that in the case of Cucker-Smale model, since $s_\rho = \phi *\rho$ and $s_\rho\aro{u} = \phi * (u\rho)$, condition (ii) holds automatically and condition (i) holds if $\phi*\rho \ge \underline{\rho}$ for some $\underline{\rho} > 0$. 
\end{remark}
\section{Proof of main result}
In this section, we will prove Theorem \ref{thm-main}. Firstly, let us introduce some notations and definitions.
\subsection{Notations and preliminaries}
The relative entropy is defined by
\begin{equation*}
    \cH(f|f_{\infty}) = \int_\dom  f \log \frac{f}{f_{\infty}} \dv \dx.
\end{equation*}
For our convenient computation, we will derive an equation for $h$ satisfying $f = h f_{\infty}$. Plugging this $f$ into equation \eqref{e:f}, we have the following equation for $h$:
\begin{equation}\label{e:h}
    \p_t h + v\cdot \n_x h = s_\rho \big( \D_vh -\n V \cdot \n_v h + h [u]_\rho \cdot \n V - [u]_\rho \cdot \n_v h \big).
\end{equation}
Letting
\[
A =  \n_v, \qquad B = v \cdot \n_x, 
\]
and $A^*$ be the adjoint of $A$ with respect to the inner product in the weighted space $L^2(\mu)$:
\[
\lan \f_1, \f_2 \ran = \int_\dom \f_1 \f_2 \dmu, \quad \dmu = f_{\infty} \dv \dx.
\]
We can calculate $A^*$ explicitly,
\[
\quad A^* = ( \n V -  \n_v) \cdot.
\]
Then we can write \eqref{e:h} in the abstract form:
\begin{equation}\label{e:FPAR}
 h_t = -s_\rho A^* A h - Bh  + s_\rho A^*(h\uave ).
\end{equation}
Following the notations from the paper \cite{RS2023env}, let us define the partial Fisher information functionals as follows:
\begin{equation*}
    \cI_{vv} (h) = \int_\dom \dfrac{|\n_v h|^2}{h}\dmu, \quad \cI_{xv}(h) = \int_\dom \dfrac{\n_x h \cdot \n_v h }{h}\dmu, \quad \cI_{xx} (h) = \int_\dom \dfrac{|\n_x h|^2}{h}\dmu.
\end{equation*}
The full Fisher information is defined by
\[
\cI = \cI_{vv} +\cI_{xx}.
\]
For our convenience we use the notation
\begin{equation*}
    (\f)_\mu := \int_\dom \f \dmu.
\end{equation*}
Denote $  \bh = \log h$ and 
\[
 \cD_{vv} = (s_\rho h|\n^2_v\bh|^2)_\mu, \quad \cD_{xv} = (s_\rho h |\n_v\n_x \bh|^2)_\mu,
\]
where $\n_v^2 \bh$  is the Hessian matrix with respect to $v$ of $\bh$. We will use the notations $J_A, J_B, J_u$ to refer to the terms related to the operators $A, B$ and related to $u$ respectively. They are different in the proof of each lemma in the sequel. We denote by $C, c$ positive constants which may vary from line to line.
\subsection{Proof of Theorem \ref{thm-main}}
By the Csisz\'ar-Kullback inequality,
\begin{equation}\label{est:CK}
    \|f-f_{\infty}\|^2_{L^1(\dom)} \le c\cH.
\end{equation}
Therefore, it suffices to show that the entropy function $\cH$ decays exponentially fast in time.
Using \eqref{e:f} and integration by parts, we have
\begin{equation}\label{Hder1}
    \ddt \cH = -\int_{\dom} s_\rho \dfrac{|\n_v f + \n V f|^2}{f} \dv \dx + \llan u_V,\uave\rran_{\k_\rho},
\end{equation}
where
\begin{equation}\label{uV}
u_V = u + u_F.
\end{equation}
Define the partial Fisher information functional $\cI_{vv}$ by
\[
\cI_{vv} = \int_{\dom} s_\rho \dfrac{|\n_v f + \n V f|^2}{f} \dv \dx.
\]
By the assumption (i) we have
\begin{equation}\label{H-aux1}
\ddt \cH \le -c_0 \cI_{vv} + \llan u_V,\uave\rran_{\k_\rho}.
\end{equation}
We can also rewrite \eqref{Hder1} in the dissipative form:
\begin{equation}\label{Hder2}
    \ddt \cH = -\int_{\dom} s_\rho \dfrac{|\n_v f + (\n V- u_V) f|^2}{f} \dv \dx - \|u_V\|^2_{L^2(\k_\rho)}+ \llan u_V,\uave\rran_{\k_\rho}.
\end{equation}
By the triangle inequality and assumption (iv) we have
\begin{equation}
    \|u\|_{L^2(\k_\rho)}  \le \dfrac{1}{1-\e_1}\|u_V\|_{L^2(\k_\rho)}.
\end{equation}
Then by the Cauchy-Schwarz inequality, assumptions (iii) and (iv) we have 
\begin{align}\label{gap}
   \llan u_V,\uave\rran_{\k_\rho} =& \llan u_V,[u_V]_{\rho}\rran_{\k_\rho} - \llan u_V,[u_F]_{\rho}\rran_{\k_\rho}\notag\\
   \le& (1-\e_0)\|u_V\|^2_{L^2(\k_\rho)} + \dfrac{\e_1}{1-\e_1}\|u_V\|^2_{L^2(\k_\rho)} \notag\\
   \le & (1 - c_3) \|u_V\|^2_{L^2(\k_\rho)},
\end{align}
where $c_3>0$ depending on $\e_0, \e_1$. Plugging this inequality into \eqref{Hder2}
 we obtain
\begin{equation}\label{H-aux2}
    \ddt \cH \le - c_3 \|u_V\|^2_{L^2(\k_\rho)}.
\end{equation}
Combining \eqref{H-aux1}, \eqref{H-aux2} and \eqref{gap} we have
\begin{equation}\label{Hdt}
    \ddt \cH \le -\dfrac{c_0c_3}{1+c_3} \cI_{vv} - \dfrac{c^2_3}{1+c_3} \|u_V\|^2_{L^2(\k_\rho)} \le -c \cI_{vv} - c \|u_V\|^2_{L^2(\k_\rho)},
\end{equation}
where $c>0$ depending on $\e_0, \e_1, c_0$.

By \eqref{HessV-coer}, 
$f_{\infty}$ satisfies a logarithmic Sobolev inequality, see \cite{villani2009hypo}. Thus, we have
\begin{equation}\label{est:log-S}
    \cH \le  c\cI.
\end{equation}
We have the following three estimates on the time derivative of partial Fisher information functionals. Their proofs will be presented in the next subsection.
\begin{lemma}\label{Ivv-deri}
We have
\begin{equation}
    \ddt \cI_{vv}(h) \le - 2 \cD_{vv} -\l c_0 \cI_{vv} -2\cI_{xv} + c\|u\|^2_{L^2(\k_\rho)},
\end{equation}
where $c$ is a positive constant depending on $c_0,c_1, \l, \L$.
\end{lemma}
\begin{lemma} \label{I2-deri} 
We have 
    \begin{equation}
        \ddt \cI_{xv} \le c\cI_{vv} - \frac12 \cI_{xx} + 2\cD_{vv} + \cD_{xv} + c\|u\|^2_{L^2(\k_\rho)},
    \end{equation}
    where $c$ is dependent on $c_0, c_1, c_2, \l, \L$.
\end{lemma}
\begin{lemma}\label{Ixx-deri}
We have 
\[
\ddt \cI_{xx}(h)  \leq   c \cI_{vv} - \cD_{xv} + c\|u\|^2_{L^2(\k_\rho)},
\]
where $c$ is a constant depending on $\l, \L$ and the parameters in the assumption (i), (ii).
\end{lemma}
Choosing $\e>0$ small so that if we define
\begin{equation}
    \tilde{\cI}  = \cI_{vv} + \e \cI_{xv} + \frac{\l c_0}{c}\cI_{xx},
\end{equation}
then $\cI \sim \tilde{\cI}$. Combining three lemmas above and the assumption (iv) we have
\begin{equation}\label{Idt}
    \ddt \tilde{\cI} \le - \l c_0 \cI_{vv} - \frac{\e}{2} \cI_{xx} + C  \|u\|^2_{L^2(\k_\rho)} \le - \l c_0 \cI_{vv} - \frac{\e}{2} \cI_{xx} + C  \|u_V\|^2_{L^2(\k_\rho)}  .
\end{equation}
From \eqref{Hdt}, \eqref{Idt} and \eqref{est:log-S} we can choose a constant $\g$ such that 
\begin{align}
    \ddt (\tilde{\cI} + \g \cH) \lesssim - \cI \le -\d(\tilde{\cI} +\g \cH).
\end{align}
Thus, by Gr\"onwall's inequality we obtain
\begin{equation}
    \tilde{\cI} + \g \cH \le (\tilde{\cI}_0 + \g \cH_0) e^{-\d t} \le c \cI_0 e^{-\d t}.
\end{equation}
Then we can conclude the theorem.
\subsection{Proof of three technical lemmas}
In this subsection, we will give the proofs of three lemmas mentioned previously.
\begin{proof}[Proof of Lemma \ref{Ivv-deri}]

Let us rewrite $\cI_{vv}$ in the form 
\[
\cI_{vv} = ( \n_v h \cdot \n_v \bh )_\mu.
\]
By chain rule and equation \eqref{e:FPAR} we get
\begin{align*}
 \ddt \cI_{vv} =&\, 2 ( \n_v h_t \cdot \n_v \bh )_\mu - ( |\n_v \bh |^2 h_t )_\mu :=  J_A + J_B + J_u ,
 \end{align*}
 where
\begin{align*}
J_A & =   -2  ( s_\rho \n_v A^*A h \cdot  \n_v \bh )_\mu + (s_\rho  |\n_v \bh |^2 A^*A h )_\mu, \\
J_B &= -2 ( \n_v B h\cdot  \n_v \bh )_\mu + ( |\n_v \bh |^2 B h )_\mu ,\\
J_u & = 2(s_\rho \n_v A^*(\uave h)\cdot \n_v \bh)_\mu - (s_\rho  |\n_v \bh |^2 A^*(\uave h) )_\mu.
\end{align*}
For notational convenience we will use the Einstein summation convention in the sequel.

We firstly consider the term $J_A$.  Using the identity
\[
\p_{v_i} (A^*A h) = A^*Ah_{v_i} + \n V_{v_i} \cdot \n_v h,
\]
$J_A$ equals to
\begin{align*}
 -2  (s_\rho  A^*Ah_{v_i} \bh_{v_i} )_\mu - 2 (s_\rho (\n V_{v_i} \cdot \n_v h)  \bh_{v_i} )_\mu   +  (s_\rho |\n_v \bh |^2 A^*A h )_\mu =: J_A^1+ J_A^2+ J_A^3.
\end{align*}
By \eqref{HessV-coer} we have
\begin{equation*}
    J_A^2 = - 2(s_\rho h^{-1} (\n_v h)^T \n^2V \n_v h)_\mu \le - 2\l (s_\rho h^{-1}\n_v h \cdot \n_v h )_\mu .
\end{equation*}
Then the assumption (i) in Theorem \ref{thm-main} implies that
\[
 J_A^2 \le  - 2\l c_0 \cI_{vv}.
\]
By switching $A^*$ in $J_A^1, J_A^3$ we can write
\[
\begin{split}
  J_A^1 + J_A^3 =& -2  (s_\rho  Ah_{v_i} \cdot A \bh_{v_i} )_\mu + (s_\rho A (|\n_v \bh |^2)\cdot  A h  )_\mu \\
   =& -2 (s_\rho  h A\bh_{v_i}\cdot  A \bh_{v_i} )_\mu 
   -2 (s_\rho \bh_{v_i} Ah\cdot  A \bh_{v_i}  )_\mu
   + 2 (s_\rho  \bh_{v_i}A \bh_{v_i} \cdot  A h  )_\mu\\
   =& -2 (s_\rho  h A\bh_{v_i}\cdot  A \bh_{v_i} )_\mu = - 2\cD_{vv}.
\end{split}
\]
Combining the above estimates we obtain
\begin{equation}\label{est:JA}
    J_A \le - 2 \cD_{vv} -2 \l c_0 \cI_{vv}.
\end{equation}

For the term  $J_B$, plugging $B = v\cdot\n_x $ into $J_B$ we have
\[
J_B = -2 ( \n_x h \cdot  \n_v \bh )_\mu - 2 ( (v \cdot \n_x h_{v_i}) \bh_{v_i} )_\mu  + ( |\n_v \bh |^2 v \cdot \n_x h )_\mu.
\]
Using the identity $\bh_{v_i} = h_{v_i}h^{-1}$ and integration by parts, we get
\[
\begin{split}
    2 ( (v \cdot \n_x h_{v_i}) \bh_{v_i} )_\mu = \big( v \cdot \n_x |h_{v_i}|^2 h^{-1} \big)_\mu = \big( |h_{v_i}|^2 h^{-2}  v \cdot  \n_x h\big)_\mu =  ( |\n_v \bh |^2 v \cdot \n_x h )_\mu.
\end{split}
\]
Substituting this into $J_B$ we yield
\begin{equation}\label{est:JB}
    J_B = - 2 \cI_{xv}.
\end{equation}

For the last term $J_u$, we have
\[
\begin{split}
    J_{u}  =&\, 2(s_\rho \n_v A^*(\uave h)\cdot  \n_v \bh )_\mu -(s_\rho |\n_v \bh |^2 A^*(\uave h))_\mu \\
 =& \,2( s_\rho \n_v ( \n V \cdot \uave h - \uave\cdot \n_v h)\cdot  \n_v \bh )_\mu- ( s_\rho \n_v |\n_v \bh |^2 \cdot  \uave h )_\mu \\
 =&\, 2(s_\rho \n^2 V(\uave h)\cdot \n_v\bh)_\mu +2 ( s_\rho (\n V \cdot \uave )\n_v h \cdot  \n_v \bh )_\mu - 2(s_\rho \n^2_v h (\uave)  \cdot  \n_v \bh )_\mu \\
 &\,- 2( s_\rho \n_v^2 \bh(\n_v \bh) \cdot \uave h )_\mu\\
=&: J_u^1 + J_u^2+J_u^3+J_u^4.
\end{split}
\]
 Plugging  
 \[
 \bh_{v_iv_j} = h^{-1} h_{v_iv_j} - h^{-2}h_{v_i}h_{v_j}
 \]
 into $J_u^4$ we get
 \[
 \begin{split}
     J_u^4 = & - 2(s_\rho h^{-1}h_{v_iv_j}\bh_{v_j}[u_i]_\rho h)_\mu  + 2(s_\rho h^{-2}h_{v_i}h_{v_j}\bh_{v_j}[u_i]_\rho h)_\mu\\
     =&  - 2(s_\rho \n^2_v h (\uave)  \cdot  \n_v \bh )_\mu + 2(s_\rho |\n_v \bh|^2 \n_v h \cdot \uave)_\mu\\
     =&\, J_u^3 + 2(s_\rho |\n_v \bh|^2 \n_v h \cdot \uave)_\mu.
 \end{split}
 \]
 Therefore,
 \[
 \begin{split}
     J_u^2+ J_u^3 + J_u^4 =& \,2 ( s_\rho (\n V \cdot \uave h)\n_v \bh \cdot  \n_v \bh )_\mu - 2(s_\rho |\n_v \bh|^2 \n_v h \cdot \uave)_\mu + 2 J_u^4\\
     = & \,2 (s_\rho A^*(\uave h)|\n_v \bh|^2)_\mu + 2 J^4_u\\
     = &\, 2 (s_\rho h\uave  \cdot A(|\n_v \bh|^2))_\mu + 2J^4_u \\
     = &\, 4 (s_\rho h\uave  \cdot \n_v^2 \bh(\n_v \bh))_\mu + 2J^4_u = 0.
 \end{split}
 \]
 Thus, 
 \begin{align}
     J_u =&\,  2(s_\rho \n^2 V(\uave h)\cdot \n_v\bh)_\mu = 2(s_\rho \n^2 V(\uave)\cdot \n_v h)_\mu\notag\\
     \le &\, 2\L c_1  \|\uave\|_{L^2(\k_\rho)} \sqrt{\cI_{vv}} \qquad \text{(by (i), \eqref{HessV-bdd} and \HI)}\notag\\
     \le &\, c \|u\|_{L^2(\k_\rho)} \sqrt{\cI_{vv}}\notag\\
     \le &\, c \|u\|^2_{L^2(\k_\rho)} + \l c_0 \cI_{vv}  \quad (\text{ by Young's inequality)}.\label{est:Ju}
 \end{align}
 Here the last constant $c$ depends on $c_0, c_1, \l, \L$.

 \medskip
Combining \eqref{est:JA}, \eqref{est:JB} and \eqref{est:Ju} we have the conclusion of this lemma.
\end{proof}
\begin{proof}[Proof of Lemma \ref{I2-deri}]
    Computing the derivative of $\cI_{xv}$ with respect to $t$ we get
\[
 \ddt \cI_{xv}(h) = (\n_x h_t \cdot \n_v \bh )_\mu + ( \n_x \bh \cdot \n_v {h}_t )_\mu - ( h_t \n_v \bh\cdot \n_x \bh )_\mu =: J_A + J_B + J_u,
\]
where 
\begin{align*}
    J_A =&   - (\n_x (s_\rho A^*Ah) \cdot \n_v \bh )_\mu - (  \n_x \bh \cdot \n_v  (s_\rho A^*Ah) )_\mu + (s_\rho  A^*Ah \n_v \bh\cdot \n_x \bh )_\mu=: J_A^1 + J_A^2 + J_A^3, \\
    J_B =& - ( \n_x (v \cdot \n_x h) \cdot \n_v \bh )_\mu - ( \n_x \bh \cdot \n_v (v\cdot \n_x h)  )_\mu + ( (v \cdot \n_x h) \n_v \bh\cdot \n_x \bh  )_\mu: = J_B^1 + J_B^2 + J_B^3,\\
    J_u =& (\n_x (s_\rho A^*(\uave h))\cdot \n_v\bh)_\mu + (\n_x \bh \cdot \n_v(s_\rho A^*(\uave h)))_\mu - (s_\rho A^* (\uave h) \n_v\bh \cdot \n_x\bh)_\mu.
\end{align*}
Let us firstly estimate $J_A$. Switching $A^*$ and using the identity 
\[
\n_v h_{x_i} = \bh_{x_i}\n_v h + h \n_v \bh_{x_i}
\]
we have
\[
\begin{split}
J_A^1 & = - ( s_\rho A^* A h_{x_i} \bh_{v_i} )_\mu - ( (s_\rho)_{x_i} A^* A h \bh_{v_i} )_\mu = - ( s_\rho \n_v h_{x_i} \cdot \n_v  \bh_{v_i} )_\mu - ( (s_\rho)_{x_i} \n_v h \cdot\n_v \bh_{v_i} )_\mu \\
& = -  ( s_\rho h \n_v \bh_{x_i} \cdot \n_v  \bh_{v_i} )_\mu- ( s_\rho  \bh_{x_i} \n_v h \cdot \n_v  \bh_{v_i} )_\mu - \left( \frac{(s_\rho)_{x_i}}{s_\rho^{1/2}} \frac{\n_v h}{h^{1/2}} \cdot  s_\rho^{1/2} h^{1/2} \n_v \bh_{v_i} \right)_\mu .
\end{split}
\]
In view of assumption (i) in Theorem \ref{thm-main}, 
\[
J_A^1 \leq  -  (s_\rho  h \n_v \bh_{x_i} \cdot \n_v  \bh_{v_i} )_\mu- (s_\rho  \bh_{x_i} \n_v h \cdot \n_v  \bh_{v_i} )_\mu + c\sqrt{\cI_{vv} \cD_{vv}},
\]
where $c>0$ is a constant depending on $c_0, c_2$.
 Next let us consider $J_A^2$. Since 
 \[
 \p_{v_i} (A^*A h) = A^*Ah_{v_i} + \n V_{v_i} \cdot \n_v h \text{ and } \n_v h_{v_i} = h\n_v \bh_{v_i} +\bh_{v_i}\n_v h,
 \]
 we have
\[
\begin{split}
J_A^2 =&  - ( s_\rho \bh_{x_i} A^*A {h}_{v_i} )_\mu - (s_\rho \bh_{x_i} \n V_{v_i} \cdot \n_v h)_\mu\\
=& - ( s_\rho \n_v\bh_{x_i} \cdot \n_v {h}_{v_i} )_\mu - (s_\rho \bh_{x_i} \n V_{v_i} \cdot \n_v h)_\mu\\
=& -(s_\rho h \n_v \bh_{x_i}\cdot \n_v \bh_{v_i})_\mu - (s_\rho\bh_{v_i} \n_v \bh_{x_i}\cdot \n_v h)_{\mu} - (s_\rho \n_x\bh \cdot (\n^2 V) (\n_v h))_{\mu}.
\end{split}
\] 
Then
\[
\begin{split}
J_A^1+J_A^2 &\le  - (s_\rho \n_x\bh \cdot (\n^2 V) (\n_v h))_{\mu} -  2(s_\rho h \n_v \bh_{x_i} \cdot \n_v  \bh_{v_i} )_\mu - (s_\rho  Ah \cdot A(\n_v \bh\cdot \n_x \bh ) )_\mu + c \sqrt{\cI_{vv} \cD_{vv}} \\
& \le - (s_\rho \n_x\bh \cdot (\n^2 V) (\n_v h))_{\mu} + 2\sqrt{ \cD_{vv} \cD_{xv}} + c \sqrt{\cI_{vv} \cD_{vv}} - J_A^3\\
& \le c_1 \L \sqrt{\cI_{vv} \cI_{xx}}+ 2\sqrt{ \cD_{vv} \cD_{xv}} + c \sqrt{\cI_{vv} \cD_{vv}} - J_A^3.
\end{split}
\]
Thus, combining all the terms of $J_A$ and applying Young's inequality we yield
\begin{align}\label{est:JAxv}
J_A & \le c_1 \L \sqrt{\cI_{vv} \cI_{xx}}+ 2\sqrt{ \cD_{vv} \cD_{xv}} + c\sqrt{\cI_{vv} \cD_{vv}},\notag\\
& \le c \cI_{vv} + \frac{1}{4}\cI_{xx} + \frac32\cD_{vv} + \cD_{xv}.
\end{align}
Now we consider $J_B$. We have
\begin{align*}
J_B^2 =&  - ( \n_x \bh \cdot \n_x h )_\mu - ( \bh_{x_i} v_j h_{x_j v_i})_\mu\\
=& - \cI_{xx} + ( \bh_{x_i x_j} v_j  h_{v_i} )_\mu\\
=& -  \cI_{xx} +( {h}_{x_i x_j} v_j \bh_{v_i} )_\mu - ( \bh_{x_i} \bh_{x_j} v_j h_{v_i} )_\mu 
= -  \cI_{xx} - J_B^1 - J_B^3.
 \end{align*}
In the last row we used the identity 
\[
\bh_{x_i x_j}= h^{-1} {h}_{x_i x_j} - \bar{ h}_{x_i} \bh_{x_j}.
\]
It follows that
\begin{equation}\label{est:JBxv}
J_B = - \cI_{xx} .
\end{equation}
Lastly let us examine $J_u$. We have
\[
\begin{split}
J_u = &\, (  (s_\rho)_{x_i} A^*(\uave h)  \bh_{v_i})_\mu + ( s_\rho A^*((\uave)_{x_i} h) \bh_{v_i})_\mu+ ( s_\rho A^*(\uave h_{x_i})  \bh_{v_i})_\mu \\
& + (s_\rho  \bh_{x_i}   A^*(\uave h_{v_i} ) )_\mu +  ( s_\rho \n_x \bh \cdot (\n^2 V)(\uave h))_\mu  - (  s_\rho h \uave  \cdot \n_v( \n_v \bh\cdot \n_x \bh) )_\mu \\
=& \,( h (s_\rho \uave)_{x_i}  \cdot \n_v \bh_{v_i})_\mu + ( s_\rho h \uave \bh_{x_i} \cdot \n_v \bh_{v_i})_\mu  + (s_\rho h \n_v \bh_{x_i} \cdot   \uave \bh_{v_i}  )_\mu\\
&  +  ( s_\rho \n_x \bh \cdot (\n^2 V)(\uave h))_\mu - (s_\rho h \uave  \cdot \n_v( \n_v \bh\cdot \n_x \bh))_\mu\\
=:&\, J_u^1 +J_u^2+ J_u^3+ J_u^4+ J_u^5.
\end{split}
\]
Since 
\[
J_u^2 + J_u^3  = (s_\rho h \uave \cdot \n_v(\n_x\bh \cdot \n_v\bh))_{\mu} = -J_u^5,
\]
we get
\[
J_u = J_u^1 + J_u^4.
\]
By the assumption (ii) in Theorem \ref{thm-main},
\[
J_u^1   = ( h (s_\rho \uave)_{x_i}  \cdot \n_v \bh_{v_i})_\mu  \leq c \| u \|_{L^2(\k_\rho)} \sqrt{\cD_{vv}}.
\]
For $J^4_u$ we use the assumption (i) and \eqref{HessV-bdd} to get
\[
\begin{split}
J^4_u &= ( s_\rho \n_x \bh \cdot (\n^2 V)(\uave h))_\mu \\
& \le c \|u\|_{L^2(\k_\rho)} \sqrt{\cI_{xx}},
\end{split}
\]
where $c$ is a constant depending on $c_1, \L$. Hence, by Young's inequality we obtain
\begin{equation}\label{est:Juxv}
    J_u \le \frac14 \cI_{xx} + \frac12\cD_{vv} + c\|u\|^2_{L^2(\k_\rho)}. 
\end{equation}
Combining three estimates \eqref{est:JAxv}, \eqref{est:JBxv} and \eqref{est:Juxv} it implies the conclusion of this lemma.
\end{proof}
 \begin{proof}[Proof of Lemma \ref{Ixx-deri}]
 Computing the derivative of $\cI_{xx}(h)$ with respect to $t$ we get
 \[
\ddt \cI_{xx}(h) =  2( \n_x h_t \cdot \n_x \bh )_\mu - ( |\n_x \bh |^2 h_t )_\mu =: J_A+ J_B + J_u,
\]
where
\begin{align*}
    J_A = &  -2 ( \n_x (s_\rho A^*A h) \cdot \n_x \bh )_\mu + (s_\rho  |\n_x \bh |^2 A^*Ah )_\mu,\\
    J_B  = & -2 ( \n_x(v \cdot \n_x h) \cdot \n_x \bh )_\mu + ( |\n_x \bh |^2 v \cdot \n_x h )_\mu, \\
    J_u = & 2( \n_x (s_\rho A^*(\uave h) ) \cdot \n_x \bh )_\mu - (s_\rho  |\n_x \bh |^2 A^*(\uave h) )_\mu .
\end{align*}
For $J_A$ we have
\[
\begin{split}
 J_A 
& =-2 ( (s_\rho)_{x_i} A h \cdot A \bh_{x_i} )_\mu -2 (s_\rho  A h_{x_i} \cdot A \bh_{x_i} )_\mu+ (s_\rho A |\n_x \bh |^2 \cdot Ah)_\mu =: J_A^1 + J_A^2+ J_A^3.
\end{split}
\]
By the assumption (i) in Theorem \ref{thm-main},
\[
J_A^1 = -2 \Big(\dfrac{ (s_\rho)_{x_i}}{s_\rho^{1/2}} \dfrac{\n_v h}{h^{1/2}} \cdot s_\rho^{1/2}h^{1/2} \n_v \bh_{x_i} \Big)_\mu\le c \sqrt{ \cI_{vv} \cD_{xv}}.
\]
Using the identity $\n_v h_{x_i} = h\n_v \bh_{x_i} + \bh_{x_i} \n_v h$, we have
\[
J_A^2 = -2 (s_\rho h \n_v \bh_{x_i} \cdot \n_v \bh_{x_i} )_\mu -2  (s_\rho  \bh_{x_i} \n_v h \cdot \n_v \bh_{x_i} )_\mu =  -2 \cD_{xv} - J_A^3.
\]
Therefore, 
\begin{equation}
  J_A \leq c  \sqrt{\cI_{vv}\cD_{xv} } -2  \cD_{xv}.  
\end{equation}
We have $J_B = 0$ because
\[
\begin{split}
 -2 ( \n_x(v \cdot \n_x h) \cdot \n_x \bh )_\mu 
 = -2 ( (v \cdot \n_x h_{x_i}) h_{x_i} h^{-1} )_\mu 
 = - ( (v \cdot \n_x |\n_x h|^2 h^{-1} )_\mu 
& =  - ( |\n_x \bh|^2 v \cdot \n_x h   )_\mu .
\end{split}
\]
For $J_u$, we have
\[
\begin{split}
    J_u = &\,  2(h (s_\rho\uave)_{x_i}   \cdot \n_v \bh_{x_i} )_\mu  +  2(s_\rho   h  \bh_{x_i} \uave \cdot \n_v \bh_{x_i} )_\mu - (s_\rho \n_v (|\n_x \bh |^2 )\cdot\uave h)_\mu\\
    = &\, 2(h (s_\rho\uave)_{x_i}   \cdot \n_v \bh_{x_i} )_\mu \\
    \le & \, c \|u\|_{L^2(\k_\rho)} \sqrt{\cD_{xv}} \qquad \text{  (by the assumption (ii) in Theorem \ref{thm-main})}.
\end{split}
\]
Combining all the estimates for $J_A, J_B$ and $J_u$ we get
\begin{equation*}
\ddt \cI_{xx}(h)  \leq    c  \sqrt{\cD_{xv} \cI_{vv}} -2 \cD_{xv}+ c \|u\|_{L^2(\k_\rho)} \sqrt{\cD_{xv}}.
\end{equation*}
Then by Young's inequality, the lemma is derived.
\end{proof}


\begin{thebibliography}{10}

\bibitem{BCC2011}
F.~Bolley, J.~A. Ca\~{n}izo, and J.~A. Carrillo.
\newblock Stochastic mean-field limit: non-{L}ipschitz forces and swarming.
\newblock {\em Math. Models Methods Appl. Sci.}, 21(11):2179--2210, 2011.

\bibitem{CamEtAl2003}
S.~Camazine, J.L. Deneubourg, N.R. Franks, J.~Sneyd, G.~Theraulaz, and
  E.~Bonabeau.
\newblock {\em Self-Organization in Biological Systems.}
\newblock Princeton University Press, United States, 2001.

\bibitem{CS2007a}
F.~Cucker and S.~Smale.
\newblock Emergent behavior in flocks.
\newblock {\em IEEE Trans. Automat. Control}, 52(5):852--862, 2007.

\bibitem{CS2007b}
F.~Cucker and S.~Smale.
\newblock On the mathematics of emergence.
\newblock {\em Jpn. J. Math.}, 2(1):197--227, 2007.

\bibitem{DesVil05}
L.~Desvillettes and C.~Villani.
\newblock On the trend to global equilibrium for spatially inhomogeneous
  kinetic systems: The {B}oltzmann equation.
\newblock {\em Invent. Math.}, 159:245--316, 2005.

\bibitem{KarperEtAl2013}
T.~K. Karper, A.~Mellet, and K.~Trivisa.
\newblock Existence of weak solutions to kinetic flocking models.
\newblock {\em SIAM J. Math. Anal.}, 45(1):215--243, 2013.

\bibitem{LRS-friction}
D.~Lear, D.~N. Reynolds, and R.~Shvydkoy.
\newblock Grassmannian reduction of {C}ucker-{S}male systems and dynamical
  opinion games.
\newblock {\em Discrete Contin. Dyn. Syst.}, 41(12):5765--, 2021.

\bibitem{ShuTa2019}
R.~Shu and E.~Tadmor.
\newblock Flocking hydrodynamics with external potentials.
\newblock {\em Arch. Ration. Mech. Anal.}, 238(1):347--381, 2020.

\bibitem{RS2022hypo}
R.~Shvydkoy.
\newblock Global hypocoercivity of kinetic {F}okker-{P}lanck-{A}lignment
  equations.
\newblock {\em Kinet. Relat. Models}, 15(2):213--237, 2022.

\bibitem{RS2023env}
R.~Shvydkoy.
\newblock Environmental averaging.
\newblock {\em arXiv 2211.00117}, 2023.

\bibitem{Sbook}
R.~Shvydkoy.
\newblock {\em Dynamics and analysis of alignment models of collective
  behavior}.
\newblock Ne\v{c}as Center Series. Birkh\"{a}user/Springer, Cham, \copyright
  2021.

\bibitem{VZ2012}
T.~Vicsek and A.~Zefeiris.
\newblock Collective motion.
\newblock {\em Physics Reprints}, 517:71--140, 2012.

\bibitem{villani2009hypo}
C.~Villani.
\newblock {\em Hypocoercivity}.
\newblock Mem. Amer. Math. Soc., 202, 2009.

\bibitem{yates2009a}
C~Yates, R~Erban, C~Escudero, I~Couzin, J~Buhl, I~Kevrekidis, P~Maini, and
  D~Sumpter.
\newblock Inherent noise can facilitate coherence in collective swarm motion.
\newblock {\em Proc. Natl. Acad. Sci. U. S. A.}, 106(14):5464--5469, 2009.

\end{thebibliography}
\end{document}